\documentclass[a4paper,11pt]{amsart}

\usepackage{comment}

\usepackage[hang,small,bf]{caption}
\usepackage[subrefformat=parens]{subcaption}
\usepackage{etoolbox}
\ifdef{\crop}{%
\usepackage[includeheadfoot,twoside=False,paperwidth=448pt,paperheight=587pt,rmargin=15pt,lmargin=15pt,tmargin=15pt,bmargin=15pt]{geometry}%
}{%
\setlength{\topmargin}{22mm}
\addtolength{\topmargin}{-1in}
\setlength{\oddsidemargin}{27mm}
\addtolength{\oddsidemargin}{-1in}
\setlength{\evensidemargin}{27mm}
\addtolength{\evensidemargin}{-1in}
\setlength{\textwidth}{156mm}
\setlength{\textheight}{240mm}
}%

\usepackage{mathtools}
\usepackage{color}
\usepackage[active]{srcltx}
\usepackage{amsmath,amsthm,amsxtra}
\usepackage{amssymb}

\usepackage[mathscr]{eucal}

\usepackage{bm}
\usepackage[all]{xy}

\usepackage{aliascnt}

\usepackage{tabularx}
\newcolumntype{C}{>{\centering\arraybackslash}X} 

\theoremstyle{plain}
\newtheorem{thm}{Theorem}[section]
\newtheorem*{thm*}{Theorem}

\newaliascnt{prop}{thm}
\newaliascnt{cor}{thm}
\newaliascnt{lem}{thm}
\newaliascnt{claim}{thm}
\newaliascnt{defn}{thm}
\newaliascnt{ques}{thm}
\newaliascnt{conj}{thm}
\newaliascnt{fact}{thm}
\newaliascnt{rem}{thm}
\newaliascnt{ex}{thm}
\newaliascnt{sett}{thm}
\newtheorem{prop}[prop]{Proposition}

\newtheorem{lem}[lem]{Lemma}

\newtheorem*{prop*}{Proposition}
\newtheorem*{cor*}{Corollary}
\newtheorem*{lem*}{Lemma}
\newtheorem*{claim*}{Claim}
\theoremstyle{definition}

\newtheorem*{defn*}{Definition}
\newtheorem*{ques*}{Question}
\newtheorem*{conj*}{Conjecture}

\newtheorem*{prob*}{Problem}

\newtheorem{rem}[rem]{Remark}

\newtheorem*{fact*}{Fact}
\newtheorem*{rem*}{Remark}
\newtheorem*{ex*}{Example}
\aliascntresetthe{prop}
\aliascntresetthe{cor}
\aliascntresetthe{lem}
\aliascntresetthe{claim}
\aliascntresetthe{defn}
\aliascntresetthe{ques}
\aliascntresetthe{conj}
\aliascntresetthe{fact}
\aliascntresetthe{rem}
\aliascntresetthe{ex}
\aliascntresetthe{sett}

\usepackage[pointedenum]{paralist}
\usepackage{varioref}
\labelformat{equation}{\textnormal{(#1)}}
\labelformat{enumi}{\textnormal{(#1)}}

\usepackage[a4paper]{hyperref}

\def\textsectionN~{\textsection{}}

\renewcommand\phi{\varphi}
\renewcommand\epsilon{\varepsilon}
\renewcommand\leq{\leqslant}
\renewcommand\geq{\geqslant}

\makeatletter
\newcommand{\set}{  \@ifstar{\@setstar}{\@set}}\newcommand{\@setstar}[2]{\{\, #1 \, ; \,  #2 \,\}}
\newcommand{\@set}[1]{\{ #1 \}}
\makeatother

\newcommand{\trans}[1][1]{\raisebox{#1ex}{\scriptsize\kern0.1em$t$\kern-0.1em}}

\newcommand{\A}{\mathbb{A}}

\DeclareMathOperator{\codim}{codim}
\DeclareMathOperator{\Hom}{Hom}

\DeclareMathOperator{\Proj}{Proj}
\DeclareMathOperator{\Spec}{Spec}

\newcommand{\Gr}{\mathrm{Gr}}

\DeclareMathOperator{\Pic}{Pic}
\DeclareMathOperator{\Cl}{Cl}

\DeclareMathOperator{\Supp}{Supp}

\DeclareMathOperator{\rank}{rank}

\DeclareMathOperator{\id}{id}

\DeclareMathOperator{\Quot}{Quot}

\DeclareMathOperator{\Sym}{Sym}

\DeclareMathOperator{\pr}{pr}

\DeclareMathOperator{\Div}{div}

\def\Z{\mathbb{Z}}
\def\Q{\mathbb{Q}}
\def\R{\mathbb{R}}
\def\C{\mathbb{C}}
\def\A{\mathbb{A}}

\def\r+{\mathbb{R}_{\geq 0}}

\def\ep{\varepsilon}

\def\r+{{\R}_{\geq 0}}
\def\q+{{\Q}_{\geq 0}}
\def\P{\mathbb{P}}

\def\*c{\C^{\times}}

\def\fm{\mathfrak{m}}
\def\<{\langle}
\def\>{\rangle}

\def\op1{\calo_{\P^1}}

\def\A{\mathbb {A}}

\def\C{\mathbb {C}}

\def\Q{\mathbb {Q}}
\def\R{\mathbb {R}}

\def\Z{\mathbb {Z}}

\newcommand{\cala}{\mathcal {A}}
\newcommand{\calb}{\mathcal {B}}
\newcommand{\calc}{\mathcal {C}}

\newcommand{\cale}{\mathcal {E}}
\newcommand{\calf}{\mathcal {F}}

\newcommand{\calo}{\mathcal {O}}

\newcommand{\calq}{\mathcal {Q}}

\newcommand{\calv}{\mathcal {V}}

\makeatletter
  
  \@addtoreset{equation}{section}
\makeatother

\title
[A remark on some punctual Quot schemes on smooth projective curves]
{A remark on some punctual Quot schemes on smooth projective curves}

\author[A.~Ito]{Atsushi~Ito}
\address{Department of Mathematics, Institute of Pure and Applied Sciences, University of Tsukuba, Tsukuba, Ibaraki 305-8571, Japan}
\email{ito-atsushi@math.tsukuba.ac.jp}

\subjclass[2020]{14H60, 14E05}
\keywords{Punctual Quot scheme, Quot-to-Chow morphism}

\begin{document}

\maketitle

\begin{abstract}
For a locally free sheaf $\mathcal{E}$ on a smooth projective curve,
we can define the punctual Quot scheme which parametrizes torsion quotients of $\mathcal{E}$ of length $n$ supported at a fixed point.
It is known that the punctual Quot scheme is a normal projective variety with canonical Gorenstein singularities.
In this note, we show that the punctual Quot scheme is a $\Q$-factorial Fano variety of Picard number one.
\end{abstract}

\section{Introduction}

Throughout this paper, we work over an algebraically closed field $k$ of any characteristic.
For a locally free sheaf $\cale$ of rank $r$ on a smooth projective curve $C$ and $n \geq 0$,
let $\Quot_C^n (\cale)$ be the Quot scheme which parametrizes torsion quotients of $\cale$ of length $n$.
It is known that $\Quot_C^n (\cale)$ is a smooth projective variety of dimension $nr$ (see \cite[Lemma 2.2, Corollary 4.7]{MR4063954} for instance). 
We can define the Quot-to-Chow morphism
\begin{align}\label{eq_quot-to-chow}
\pi : \Quot_C^n (\cale) \to \Sym^n C  
\end{align}
sending the quotient $[\cale \twoheadrightarrow \calq]$ to the effective divisor on $C$ determined by the torsion sheaf $\calq$.
For $q \in C$, the \emph{punctual Quot scheme} $\Quot_C^n (\cale)_q$ is defined to be the scheme-theoretic fiber of $\pi$ over $nq \in  \Sym^n C  $.
Recently, the fibers of $\pi$ are studied by many authors. For example, the following are known:

\begin{itemize}
\setlength{\itemsep}{0mm} 
\item (\cite[\S 2.1]{MR4175444}) 
The isomorphism class of $\Quot_C^n (\cale)_q$ depends only on $r$ and $n$.
In particular, it is independent of $C$ and $q$.
\item (\cite[Lemma 6.5]{MR4124833}, \cite[Theorem 1.2]{Birkar:2024aa}) 
The fiber of $\pi$ over $\sum_{i=1}^l m_i q_i \in  \Sym^n C   $ with $q_i \neq q_j  \ (i \neq j)$  is isomorphic to the product $\prod_{i=1}^l \Quot_C^{m_i} (\cale)_{q_i} $. 
\item (\cite[Corollary 6.6]{MR4124833}, \cite[Theorem 1.2]{Birkar:2024aa}) $\Quot_C^n (\cale)_q$ is a normal projective variety of dimension $n(r-1)$ with Cartier canonical divisor.
\item (\cite[\S 4,5]{MR4124833}) $\Quot_C^n (\cale)_q$ is birational to $\P^{n(r-1)}$.
\item (\cite[Lemma 6.2]{MR4742855}) $\Quot_C^n (\cale)_q$ has a crepant resolution.
In particular, $\Quot_C^n (\cale)_q$ has canonical Gorenstein singularities.
\end{itemize}

If $r=1$, \ref{eq_quot-to-chow} is an isomorphism and hence $\Quot_C^n (\cale)_q$ is a point.
If $n=1$, \ref{eq_quot-to-chow} coincides with the $\P^1$-bundle $\P_C(\cale) \to C$ 
and hence  $\Quot_C^1 (\cale)_q =\P(\cale \otimes k(q)) \simeq \P^{r-1}$.

In \cite{Birkar:2024aa},
the authors investigate the geometry of $\Quot_C^n (\cale)_q $ for $r=2$ in detail.
In particular,
they prove that $\Quot_C^n (\calo_C^{\oplus 2})_q$ is 
\begin{itemize}
\item $\P^1$ if $n=1$, 
\item  a singular quadric in $\P^3$ if $n=2$, 
\item a normal $\Q$-factorial Fano $3$-fold of Picard number one with canonical singularities along a copy of $\P^1$  if $n=3$
\end{itemize}
in characteristic zero \cite[Theorems 1.3, 1.4, 1.5]{Birkar:2024aa}.
The purpose of this note is to show a similar statement for any $n,r$ as follows.

\begin{thm}\label{thm_Picard_number_one}
Let $\cale$ be  a locally free sheaf on a smooth projective curve $C$ of rank $r \geq 2$ and $q \in C$.
For $n\geq 1$, the following hold.
\begin{enumerate}
\item 
$\Quot_C^n (\cale)_q$ is a normal $\Q$-factorial Fano $n(r-1)$-fold of Picard number one.
\item 
There exists an embedding $ \Quot_C^n (\cale)_q \hookrightarrow \Gr(nr,n) $ to a Grassmannian 
such that $\calo(1) \coloneqq\calo_{\Gr(nr,n) }(1)|_{\Quot_C^n (\cale)_q} $ is the ample generator of the Picard group $ \Pic (\Quot_C^n (\cale)_q) \simeq \Z$.
\item The Fano index of $ \Quot_C^n (\cale)_q$ is $r$, that is, $K_{\Quot_C^n (\cale)_q} = \calo(-r) $.
\item If $n \geq 2$, the singular locus of $ \Quot_C^n (\cale)_q$ is irreducible of codimension two in $ \Quot_C^n (\cale)_q$.
\item There exists a prime divisor $H \subset \Quot_C^n (\cale)_q$ such that the divisor class group $\Cl (\Quot_C^n (\cale)_q )$ is generated by the class $[H]$ and 
$n H \sim \calo(1)  $.
\end{enumerate}
\end{thm}

The idea of the proof essentially follows from \cite{MR4124833},
where the authors construct a resolution of $\Quot_C^n (\cale)_q$ as an iterated $\P^{r-1}$-bundle.
Following their construction,
we define a $\P^{r-1}$-bundle $f_n : \P_{\Quot_C^n (\cale)_q}(\calf) \to \Quot_C^n (\cale)_q$
and a divisorial contraction $\mu_{n+1} : \P_{\Quot_C^n (\cale)_q}(\calf) \to \Quot_C^{n+1} (\cale)_q$.
Then we can prove the theorem by the induction on $n$.

\vspace{2mm}
This paper is organized as follows. In  \S \ref{sec_preliminaries}, we recall some notation and give an embedding of $ \Quot_C^{n} (\cale)_q$ to  a Grassmannian. 
In \S \ref{sec_picard_group} and \S \ref{sec_div_class/group}, we investigate the Picard group and  the divisor class group of $\Quot_C^n (\cale)_q $ respectively.
In \S \ref{sec_r=2}, we give a description of the exceptional divisor of the divisorial contraction $\mu_{n+1} : \P_{\Quot_C^n (\cale)_q}(\calf) \to \Quot_C^{n+1} (\cale)_q$ for $r=2$.

\subsection*{Acknowledgments}
The author was supported by JSPS KAKENHI Grant Number 
21K03201.

\section{Embedding to a Grassmannian}\label{sec_preliminaries}

For a $k$-vector space $E$,
$\Gr(E,s)$ (resp.\ $\Gr(s,E)$) 
denotes the Grassmannian of $s$-dimensional quotients (resp.\ subspaces) 
of $E$.
More generally, for a locally free sheaf $\cale$ on a variety $X$,
$\Gr_X(\cale,s)$ (resp.\ $\Gr_X(s,\cale)$) 
denotes the Grassmannian bundle
which parametrizes quotient bundles  (resp.\  subbundles)  of $\varphi^* \cale$ of rank $s$
for each $\varphi : T \to X$.
We use the notation $\P(E) \coloneqq \Gr(E,1) =\Proj \Sym E$ and $\P_X(\cale) \coloneqq \Gr_X(\cale,1) = \Proj_X \Sym \cale$.

For a coherent sheaf of $\calf$ on $X$,
$\Quot_X^n (\calf)$ denotes the Quot scheme  which parametrizes quotients of $\calf$ with zero-dimensional supports of degree $n$.
The point in $\Quot_X^n (\calf)$ corresponding to an exact sequence $0 \to \cala \to \calf \to \calb \to 0$ on $C$ is denoted by
$[\cala \hookrightarrow \calf]$ or $[\calf \twoheadrightarrow \calb]$.
If the context is clear, we write it as $[\cala]$ or $[\calb]$ simply .

\vspace{2mm}
Let  $\cale$ be a locally free sheaf of rank $r$ on a smooth projective curve $C$ and $n \geq 0$.
Throughout this paper,
$p_C :  C \times T \to C$ and $p_T :  C \times T \to T$ are the natural projections for a locally noetherian scheme $T$ over $k$.
Then a morphism $T \to \Quot_{C}^n (\cale)$ corresponds to an exact sequence 
\begin{align}\label{eq_exact}
0 \to \mathscr{A} \to p_C^* \cale \to \mathscr{B} \to 0
\end{align}
on $C \times T$ such that $\mathscr{B}$ is locally free of rank $n$ as an $\calo_T$-module.
Since $C$ is a smooth curve, $\mathscr{A}$ is locally free of rank $r$.

Recall the definition of 
the Quot-to-Chow morphism $\pi : \Quot_C^n (\cale) \to \Sym^n C$ (see \cite[Section 2]{MR4124833} for the details).
Let $Q= \Quot_C^n (\cale)$ and let $0 \to \mathscr{A}_Q \to p_C^* \cale \to \calb_Q \to 0$ be the universal  exact sequence on $C \times Q$.
Since $\mathscr{A}_Q $ is locally free of rank $r=\rank \cale$, we obtain an exact sequence 
\begin{align*}
0 \to \det \mathscr{A}_Q  \to  \det p_C^* \cale \to \mathscr{C} \to 0.
\end{align*}
We can check that $\mathscr{C}$ is flat over $Q$ and hence 
\[
0 \to \det \mathscr{A}_Q  \otimes (\det p_C^* \cale)^{-1}  \to \calo_{C \times Q} \to \mathscr{C}  \otimes (\det p_C^* \cale)^{-1} \to 0
\]
induces
the Quot-to-Chow morphism $ \pi  : Q=\Quot_C^n (\cale) \to \Sym^n C$.

For $q \in C$, let $\Quot_{C}^n (\cale)_q $ be the scheme theoretic fiber of $\pi : \Quot_C^n (\cale) \to \Sym^n C$ over $nq$.
By the definition of $\pi$,
the morphism $T \to \Quot_{C}^n (\cale) $ corresponding to \ref{eq_exact}
factors through $\pi^{-1}(nq) =\Quot_{C}^n (\cale)_q \subset \Quot_{C}^n (\cale) $ if and only if $\det \mathscr{A} = p_C^* \fm^n \det \cale$,
where $\fm=\calo_C(-q)$ is the maximal ideal sheaf corresponding to $q \in C$.

The following proposition is essentially explained in \cite[\S 6.4]{Birkar:2024aa}, at least set-theoretically.

\begin{prop}\label{lem_embedding_to_Grassmannian}
Under the above setting,
$\Quot_{C}^n (\cale)_q$ coincides with $\Quot_C^n (\cale/\fm^n \cale)_{\mathrm{red}} \subset \Quot_{C}^n (\cale)$,
where we embed\footnote{See \cite[\S 5.5.3]{MR2222646} for the embedding between Quot schemes induced by a surjection of coherent sheaves.} $\Quot_C^n (\cale/\fm^n \cale) $ to $ \Quot_{C}^n (\cale) $ by the natural surjection $ \cale \to \cale/\fm^n \cale$
 and 
$\Quot_C^n (\cale/\fm^n \cale)_{\mathrm{red}} $ means the reduced scheme structure on $\Quot_C^n (\cale/\fm^n \cale)$.

In particular,
there exists an embedding $\Quot_{C}^n (\cale)_q = \Quot_C^n (\cale/\fm^n \cale)_{\mathrm{red}}  \hookrightarrow  \Gr(\cale/ \fm^n\cale, n ) =\Gr(nr,n) $.
\end{prop}

\begin{proof}
Consider a morphism $T \to \Quot_C^n (\cale)_q$, which corresponds to an exact sequence  $0 \to \mathscr{A} \to p_C^* \cale \to \mathscr{B} \to 0$ on $C \times T$.
Then $ \mathscr{A}$ is locally free of rank $r$ with $\det \mathscr{A} =  p_C^* \fm^n \det  \cale$.
By Cramer's rule,  $\mathscr{A} $ contains $  p_C^* \fm^n \cale$
and hence the morphism $T \to \Quot_C^n (\cale)_q \subset \Quot_C^n (\cale)$ factors through the subscheme
$ \Quot_C^n (\cale/\fm^n \cale) \subset \Quot_C^n (\cale)$.
This means that $ \Quot_C^n (\cale)_q $ is a subscheme of $ \Quot_C^n (\cale/\fm^n \cale)$.

A closed point of $  \Quot_C^n (\cale/\fm^n \cale)$ is a quotient $[\cale/\fm^n \cale \to \calb] $ on $C$ whose length is $n$.
As a point in $ \Quot_C^n (\cale)$,
this is the point  $[\cale \to \cale/\fm^n \cale \to \calb]$, which is contained in $ \Quot_C^n (\cale)_q $  since $\Supp(\calb) =\{q\}$.
Since $ \Quot_C^n (\cale)_q $ is reduced by \cite{MR4124833} or \cite{Birkar:2024aa}, 
$\Quot_{C}^n (\cale)_q = \Quot_C^n (\cale/\fm^n \cale)_{\mathrm{red}} $ holds.

Since $\cale/\fm^n \cale$ is a $k$-vector space of dimension $nr$, 
the Quot scheme $ \Quot_C^n (\cale/\fm^n \cale)$ is naturally embedded to the Grassmannian $\Gr(\cale/\fm^n \cale, n)=\Gr(nr,n)$.
In fact,  a morphism $T \to \Quot_C^n (\cale/\fm^n \cale)$
corresponds to a quotient  $( \cale/\fm^n \cale) \otimes \calo_T  \to \mathscr{B}$ of $ \calo_{C,q} \otimes \calo_T $-modules such that  $\mathscr{B}$ is locally free of rank $n$ as $\calo_T$-module.
On the other hand,
a morphism $T \to \Gr(\cale/\fm^n \cale, n)$
corresponds to a quotient  $( \cale/\fm^n \cale) \otimes \calo_T  \to \mathscr{B}$ of $\calo_T $-modules such that  $\mathscr{B}$ is locally free of rank $n$. 
Hence there exists a natural injection
$\Hom_{k\text{-}sch} (T,  \Quot_C^n (\cale/\fm^n \cale)) \to \Hom_{k\text{-}sch} (T,  \Gr(\cale/\fm^n \cale, n))$.
Thus $ \Quot_C^n (\cale/\fm^n \cale)$ is embedded into  $\Gr(\cale/\fm^n \cale, n)=\Gr(nr,n)$.
\end{proof}

\begin{rem}\label{rem_embedding_to_Grassmannian}
Let $ 0 \to \mathscr{A}_n \to p_C^* \cale \to \mathscr{B}_n \to 0$ be the universal exact sequence on $C \times \Quot_{C}^n (\cale)_q$.
The embedding  $\Quot_{C}^n (\cale)_q \to \Gr(\cale/ \fm^n\cale, n )$ is induced by
\begin{align*}
0 \to \overline{\mathscr{A}_n/ p_C^* \fm^n \cale} \to \overline{p_C^* \cale /p_C^* \fm^n \cale} = ( \cale / \fm^n \cale ) \otimes \calo_{ \Quot_{C}^n (\cale)_q}  \to \overline{\mathscr{B}_{n}} \to 0,
\end{align*}
where $\overline{\calf}$ is the pushforward of $\calf $ by  $p_{\Quot_{C}^n (\cale)_q} : C \times \Quot_{C}^n (\cale)_q \to \Quot_{C}^n (\cale)_q$.
In particular,
 $ \calo_{ \Gr(\cale/ \fm^n\cale, n ) } (1) |_{ \Quot_{C}^n (\cale)_q}  = \det  \overline{\mathscr{B}_{n}}$ holds,
 where $ \calo_{ \Gr(\cale/ \fm^n\cale, n ) } (1)$ is the Pl\"{u}cker line bundle of the Grassmannian $ \Gr(\cale/ \fm^n\cale, n )$.
\end{rem}

\begin{rem}\label{rem_pic_Q}
The Picard group of the Quot scheme $Q= \Quot_{C}^n (\cale) $ is computed by \cite{MR4323001} as follows.
Let $ 0 \to \mathscr{A}_Q \to p_C^* \cale \to \mathscr{B}_Q \to 0$ be the universal exact sequence on $C \times Q$
and let $ \calo_Q (1) = \det ( {p_Q}_* ( \mathscr{B}_Q) )$.
Then $\pi^* : \Pic^0 (\Sym^n C ) \to \Pic (Q)$ induced by 
the Quot-to-Chow morphism $\pi : Q \to \Sym^n C  $ is injective and 
$\Pic (Q) = \pi^*\Pic^0 (\Sym^n C ) \oplus \Z[ \calo_Q (1)] $ by  \cite[Theorem 3.7]{MR4323001}.

Then $ \calo_{ \Gr(\cale/ \fm^n\cale, n ) } (1) |_{ \Quot_{C}^n (\cale)_q}$ coincides with    $ \calo_Q (1) |_{ \Quot_{C}^n (\cale)_q}$
since 
\[
\calo_{ \Gr(\cale/ \fm^n\cale, n ) } (1) |_{ \Quot_{C}^n (\cale)_q} =\det  \overline{\mathscr{B}_{n}} = \det( {p_Q}_* ( \mathscr{B}_Q)|_{\Quot_{C}^n (\cale)_q}) =  \calo_Q (1)|_{\Quot_{C}^n (\cale)_q}
\] for
$ \overline{\mathscr{B}_{n}} $ in \autoref{rem_embedding_to_Grassmannian}. 
\end{rem}

\begin{rem}\label{rem_non-reduced}
In general, $  \Quot_C^n (\cale/\fm^n \cale)$ is non-reduced.
For example, let $ \cale=\calo_{C}^{\oplus 2}, n =2$.
Then $\cale/\fm^2 \cale =(k[t]/(t^2))^{\oplus 2}$,
where $t$ is a local coordinate of $C$ at $q$.
For $T= \Spec R=\Spec k[\ep]/(\ep^2)$,
a quotient
\begin{align*}
 p_C^*( \cale/\fm^2 \cale) = (R[t]/(t^2))^{\oplus 2} \to R^{\oplus 2} : (f(t),g(t)) \mapsto (f(\epsilon),g(\epsilon)) 
\end{align*}
on $ C \times T$
gives a morphism $ T \to \Quot_{C}^2 (\cale/\fm^2 \cale)$.
This does not factor through $\Quot_{C}^2 (\cale)_q $ if $\mathrm{char} \ k \neq 2$ since the kernel of $ p_C^* \cale=\calo_{C \times T}^{\oplus 2} \to  (R[t]/(t^2))^{\oplus 2} \to R^{\oplus 2}$
is $(t-\ep) \calo_{C \times T}^{\oplus 2}$, whose determinant is $(t^2 -2\ep t) \calo_{C \times T} \neq t^2 \calo_{C \times T}$.
\end{rem}

\section{Picard groups}\label{sec_picard_group}

Throughout this section, we fix a locally free sheaf $\cale$ of rank $r$ on a smooth projective curve $C$  and $q \in C$.
Since the punctual Quot scheme $\Quot_{C}^n (\cale)_q$ is a point if $r=1$,
we assume $r \geq 2$ in the rest of this section. 
For simplicity, we set $F_n=\Quot_{C}^n (\cale)_q$ and $\fm =\calo_C(-q) \subset \calo_C$.
As in \autoref{rem_embedding_to_Grassmannian},
$\overline{\calf}$ denotes the pushforward of a coherent sheaf $\calf $ on $C \times F_n$ by the projection $p_{F_n} : C \times  F_n \to F_n$.

Let 
\begin{align*}
0 \to \mathscr{A}_n \to p_C^* \cale \to \mathscr{B}_n \to 0
\end{align*}
be the universal exact sequence on $C \times F_n$.
Recall that
$\mathscr{A}_n$ is locally free of rank $r$ with $\det \mathscr{A}_n = p_C^*  \mathfrak{m}^n \det \cale$.
Then $\overline{\mathscr{A}_n / \fm \mathscr{A}_n}= \mathscr{A}_n|_{ \{q\} \times F_n} $ 
is locally free of rank $r$ on $F_n$ and hence we can define
a $\P^{r-1}$ bundle $f_n :  \P_{F_n}( \overline{\mathscr{A}_n / \fm \mathscr{A}_n}) \to F_n$.
Let $f_n^* (\overline{\mathscr{A}_n/\fm \mathscr{A}_n}) \to \calo_{f_n}(1) $ be the tautological line bundle on $ \P_{F_n}(\overline{ \mathscr{A}_n / \fm \mathscr{A}_n}) $.

\begin{lem}\label{lem_F_n,n+1=P}
For $n \geq 0$, $ \P_{F_n}(\overline{ \mathscr{A}_n / \fm \mathscr{A}_n})  $ is isomorphic to 
\begin{align*}
F_{n,n+1} &\coloneqq \{([\cala_n],[\cala_{n+1}]) \in F_n \times F_{n+1} \mid \cala_{n+1} \subset \cala_n \subset \cale\}
\end{align*}
with the reduced structure
over $F_n$.
\end{lem}

\begin{proof}
Let $\pr_n : F_{n,n+1} \to F_n$ and $\pr_{n+1} : F_{n,n+1} \to F_{n+1}$ be the natural projections.
We first explain this lemma set-theoretically.
Fix $[\cala_n] \in F_n$.
Then a point in $\pr_n^{-1}([\cala_n]) $ corresponds to a quotient $\calo_{C}$-module $\cala_n \to \calv$ of length one with $\Supp \calv =\{q\}$.
Since such $\calv$ is isomorphic to $\calo_C/\fm$,
such quotient $\cala_n \to \calv$ corresponds to a quotient $k$-vector space $\cala_n/\fm \cala_n \to V$ with $\dim_k V=1$,
which is nothing but a point in $\P( \cala_n/ \fm \cala_n)  =\P( \overline{\mathscr{A}_n / \fm \mathscr{A}_n} \otimes k([\cala_n])) =f_n^{-1}([\cala_n]) $.
Hence there exists a canonical bijection between $ F_{n,n+1} $ and  $ \P_{F_n}(\overline{ \mathscr{A}_n / \fm \mathscr{A}_n})  $.

We can construct this bijection as an isomorphism as follows.
For simplicity, $\P_{F_n}$ denotes $ \P_{F_n}(\overline{ \mathscr{A}_n / \fm \mathscr{A}_n})  $.
Let 
\[
\iota :  \P_{F_n}=  \{q\}  \times \P_{F_n} \hookrightarrow  C \times \P_{F_n}
\]
be the natural immersion.
Since $\left((\id_C \times f_n)^* \mathscr{A}_n \right) |_{\{q \} \times  \P_{F_n}}= (\id_C \times f_n)^*  (\mathscr{A}_n|_{\{q \} \times F_n})
=\iota_* f_n^* (\overline{ \mathscr{A}_n / \fm \mathscr{A}_n})$,
we can consider the composite map
\begin{align}\label{eq_A'}
(\id_C \times f_n)^* \mathscr{A}_n \to \iota_* f_n^* (\overline{\mathscr{A}_n/\fm \mathscr{A}_n}) \to \iota_* \calo_{f_n}(1) 
\end{align}
on $C \times   \P_{F_n} $.
Let $\mathscr{A}'$ be the kernel of \ref{eq_A'}.
Then we have a diagram
\[
\xymatrix{
  0 \ar[r] & \mathscr{A}' \ar[r] \ar@{^(->}[d] & p_C^* \cale \ar[r] \ar@{=}[d] & p_C^* \cale / \mathscr{A}' \ar[r] \ar@{>>}[d] & 0 \\
  0 \ar[r] & (\id_C \times f_n)^* \mathscr{A}_n \ar[r] & p_C^* \cale \ar[r] & (\id_C \times f_n)^* \mathscr{B}_n \ar[r] & 0
}
\]
on $C \times  \P_{F_n} $.
By the snake lemma,
we have an exact sequence 
\begin{align}\label{eq_O(1)_A'_B_n}
0 \to  \iota_* \calo_{f_n}(1)  \to p_C^* \cale/ \mathscr{A}'   \to (\id_C \times f_n)^*\mathscr{B}_{n} \to 0.
\end{align}
Since $  \iota_* \calo_{f_n}(1)  $ and $ (\id_C \times f_n)^*\mathscr{B}_{n}$ are flat over $ \P_{F_n} $,
so is $ p_C^* \cale/ \mathscr{A}'  $.
Since $ \mathscr{A}' $ is the kernel of \ref{eq_A'},
we see that $\mathscr{A}'$ is a locally free of rank $r$ with $\det \mathscr{A}' = \fm (\id_C \times f_n)^* \det \mathscr{A}_n= p_C^* \mathfrak{m}^{n+1} \det \cale$.
Hence $0 \to \mathscr{A}' \to p_C^* \cale \to p_C^* \cale/ \mathscr{A}'   \to 0$ induces a morphism $ \mu_{n+1} :  \P_{F_n}   \to F_{n+1} =\Quot_{C}^{n+1} (\cale)_q$.
Since $\mathscr{A}' \subset (\id_C \times f_n)^* \mathscr{A}_n$, 
the image of $f_n \times \mu_{n+1} : \P_{F_n}  \to F_n \times F_{n+1}$ is contained in $F_{n,n+1}$.

The inverse of $f_n \times \mu_{n+1}$ is constructed as follows.
For simplicity, $\id_C \times \pr_i  : C \times F_{n,n+1} \to C \times F_i$
 is denoted by $\tilde{\pr}_i$.
Then we have a diagram
\[
\xymatrix{
  0 \ar[r] & \tilde{\pr}_{n+1}^* \mathscr{A}_{n+1} \ar[r] \ar@{^(->}[d] & p_C^* \cale \ar[r] \ar@{=}[d] & \tilde{\pr}_{n+1}^* \mathscr{B}_{n+1} \ar[r] \ar@{>>}[d]^{\varpi} & 0 \\
  0 \ar[r] & \tilde{\pr}_{n}^* \mathscr{A}_{n} \ar[r] & p_C^* \cale \ar[r] & \tilde{\pr}_{n}^* \mathscr{B}_{n} \ar[r] & 0
}
\]
on $C \times F_{n,n+1}$.
Hence 
\begin{align}\label{eq_F_n,n+1=P}
\overline{ \tilde{\pr}_{n}^* \mathscr{A}_{n} /  \tilde{\pr}_{n+1}^* \mathscr{A}_{n+1} } \simeq \overline{ \ker \varpi} 
\end{align}
is locally free of rank one on $F_{n,n+1}$.
Since $  \det \tilde{\pr}_{n+1}^* \mathscr{A}_{n+1} =  p_C^*\fm^{n+1} \det \cale  = \fm \det \tilde{\pr}_{n}^* \mathscr{A}_{n} $,
we have $\fm  \tilde{\pr}_{n}^* \mathscr{A}_{n}  \subset  \tilde{\pr}_{n+1}^* \mathscr{A}_{n+1}$ by Cramer's rule
and hence \ref{eq_F_n,n+1=P} is a quotient of
\begin{align*}
\overline{ \tilde{\pr}_{n}^* \mathscr{A}_{n} /  \fm  \tilde{\pr}_{n}^* \mathscr{A}_{n} }  = \pr_n^* (\overline{ \mathscr{A}_n / \fm \mathscr{A}_n}).
\end{align*}
Hence $\pr_n^* (\overline{ \mathscr{A}_n / \fm \mathscr{A}_n}) \to  \overline{ \tilde{\pr}_{n}^* \mathscr{A}_{n} /  \tilde{\pr}_{n+1}^* \mathscr{A}_{n+1} }$
induces a morphism $F_{n,n+1} \to  \P_{F_n} $.
By construction,
this is the inverse of $f_n \times \mu_{n+1} :  \P_{F_n}  \to F_{n,n+1}$.
\end{proof}

By \autoref{lem_F_n,n+1=P}, $\pr_n : F_{n,n+1} \to F_n$ is a $\P^{r-1}$-bundle.
On the other hand, $\pr_{n+1} : F_{n,n+1} \to F_{n+1}$ is birational as follows.

\begin{lem}\label{lem_fiber_F_n,n+1_to_F_n+1}
Set $U \coloneqq \{ [\cale \twoheadrightarrow \calb_{n+1}] \in F_{n+1} \mid \calb_{n+1} \simeq \calo_C/\fm^{n+1}\}$. Then
\begin{enumerate}
\setlength{\itemsep}{0mm}
\item $\pr_{n+1} : F_{n,n+1} \to F_{n+1}$ is an isomorphism over $U $. 
\item The dimension of the fiber of $\pr_{n+1}$ over a point in $F_{n+1} \setminus U$ is positive.
\item The codimension of $ \pr_{n+1}^{-1} (F_{n+1} \setminus U) $ in $ F_{n,n+1}$ is one.
\end{enumerate}
\end{lem}

\begin{proof}
Let $[\cala_{n+1}] \in F_{n+1}$ be a point corresponding to $0 \to \cala_{n+1} \to \cale \to \calb_{n+1} \to 0$. 
Let $ \calb'_{n+1}=\{ b \in  \calb_{n+1} \mid \fm b =0 \} $. 
Then the fiber $\pr_{n+1}^{-1}([\cala_{n+1}])$ is canonically identified with $\Gr(1, \calb'_{n+1})  $ as follows:
A point in the fiber $\pr_{n+1}^{-1} ([\cala_{n+1}])$ corresponds to a quotient $\calb_{n+1} \to \calb_n$ of length $n$.
Such quotient corresponds to a submodule $\calc \subset \calb_{n+1} $ of length one.
Since such a submodule $\calc $ is isomorphic to $\calo_C/\fm$ and hence $\fm \calc=0$,
a submodule $\calc \subset \calb_{n+1} $ of length one is nothing but a one dimensional subspace of $\calb'_{n+1}$.
Hence there exists a bijection between the fiber $\pr_{n+1}^{-1} ([\cala_{n+1}])$ and  $\Gr(1,\calb'_{n+1})$.
\\
(1) If $ \calb_{n+1} \simeq \calo_C/\fm^{n+1}$, it holds that $\calb'_{n+1} =\fm^n \calb_{n+1} $ and $ \calb_{n+1} / \calb'_{n+1} \simeq \calo_C/\fm^{n}$.
Hence for the universal exact sequence $ 0 \to \mathscr{A}_{n+1} \to p_C^* \cale \to \mathscr{B}_{n+1}  \to 0$ on $C \times F_{n+1}$,
the quotient $\mathscr{B}_{n+1} /\fm^n \mathscr{B}_{n+1} $ is flat of length $n$ over $U$.
Hence $p_C^* \cale|_{C \times U}  \to (\mathscr{B}_{n+1} /\fm^n \mathscr{B}_{n+1})|_{C \times U} $ gives a morphism $g : U \to F_n$.
By construction, $(g, \id_{U} ) : U \to F_{n,n+1}$ is the inverse of $\pr_{n+1}$ over $U$.
\\
(2) If $ \calb_{n+1} \not \simeq \calo_C/\fm^{n+1}$,
it holds that $ \calb_{n+1} \simeq \bigoplus_{i=1}^l \calo_C/\fm^{n_i} $ with $\sum_{i=1}^l n_i=n+1$ for some $ l\geq 2$ and $n_i \geq 1$
by the classification of modules over PID.
Then $\calb'_{n+1} \simeq \bigoplus_{i=1}^l \fm^{n_i-1}/\fm^{n_i} \simeq k^l$ and hence the dimension of $\Gr(1,\calb'_{n+1}) =\P^{l-1}$ is positive.\\
(3) Since $\codim_{F_{n+1}} (F_{n+1} \setminus U) =2$ by \cite[\S 5]{Birkar:2024aa},
the codimension of the exceptional locus $\pr_{n+1}^{-1}(F_{n+1} \setminus U) $ is one
by (1), (2) and the irreducibility of $F_{n,n+1}$.
\end{proof}

\begin{prop}\label{prop_Picard_number}
For each $n\geq 1$, the following hold.
\begin{enumerate}
\setlength{\itemsep}{0mm}
\item $F_n$ is $\Q$-factorial and $\Pic F_n =\Z \calo_{F_n}(1)$, where $\calo_{F_n} (1) =  \calo_{\Gr(\cale/\fm^n \cale,n)} (1) |_{F_n} $ is the restriction of $ \calo_{\Gr(\cale/\fm^n \cale,n)} (1)$ to  $F_n \subset \Gr(\cale/\fm^n \cale,n)$.
\item $K_{F_n} = \calo_{F_n}(-r)$ and $K_{F_{n,n+1}} =\pr_{n+1}^* K_{F_n}$ hold. 
\item $\pr_{n+1} : F_{n,n+1} \to F_{n+1}$ is a divisorial contraction.
\item For $n \geq 2$, the singular locus of $F_{n}$ is $\{[\cale \twoheadrightarrow \calb_n ] \in F_n \mid \calb_n \not \simeq \calo_C/\fm^n\}$,
which is irreducible of codimension two in $F_n$.
\end{enumerate}
\end{prop}

\begin{proof}
We show (1) and (2) by the induction of $n$.
Since $F_1 =\P(\cale / \fm \cale) \simeq \P^{r-1}$, (1), (2) hold for $n=1$.
We assume (1), (2) for $n$
and show (1), (2) for $n+1$.

By induction hypothesis,
$\P_{F_n}( \overline{\mathscr{A}_n / \fm \mathscr{A}_n})  =F_{n,n+1}$ is $\Q$-factorial with 
$\Pic \P_{F_n}( \overline{\mathscr{A}_n / \fm \mathscr{A}_n}) =  \Z  \calo_{f_n}(1) \oplus \Z f_n^* \calo_{F_n} (1)$,
where $\calo_{f_n}(1) $  is the tautological line bundle of  $f_n :  \P_{F_n}( \overline{\mathscr{A}_n / \fm \mathscr{A}_n})  \to F_n$.
Since $\pr_{n+1} : F_{n,n+1} \to F_{n+1}$ is birational and contracts a divisor by \autoref{lem_fiber_F_n,n+1_to_F_n+1},
$F_{n+1}$ is $\Q$-factorial with Picard number one.

Recall that the embedding $ F_i \hookrightarrow  \Gr(\cale/\fm^{i}\cale,i)$ is induced by the quotient 
\begin{align*}
(\cale/\fm^{i} \cale)\otimes \calo_{F_i}=  \overline{p_C^* \cale/p_C^* \mathfrak{m}_q^{i} \cale}  \to \overline{\mathscr{B}_i} 
\end{align*}
on $F_i$ and hence $ \calo_{F_i}(1) =\det  \overline{\mathscr{B}_i}$ by \autoref{rem_embedding_to_Grassmannian} for $i =n,n+1$.
On the other hand,
 $\pr_{n+1} : \P_{F_n}( \overline{\mathscr{A}_n / \fm \mathscr{A}_n})=F_{n,n+1} \to F_{n+1} $ is induced by the quotient
$ p_C^* \cale \to p_C^* \cale/\mathscr{A}' $ on $C \times  \P_{F_n}( \overline{\mathscr{A}_n / \fm \mathscr{A}_n})$,
where $\mathscr{A}' $ in the kernel of \ref{eq_A'}.
Hence $\pr_{n+1} : \P_{F_n}( \overline{\mathscr{A}_n / \fm \mathscr{A}_n})=F_{n,n+1} \to F_{n+1} \subset \Gr(\cale/\fm^{n+1}\cale,n+1)$ is induced by 
the quotient 
\begin{align*}
(\cale/\fm^{n+1} \cale)\otimes \calo_{ \P_{F_n}( \overline{\mathscr{A}_n / \fm \mathscr{A}_n})} =  {p}_*(p_C^* \cale/p_C^* \mathfrak{m}_q^{n+1} \cale)   \to p_*(p_C^* \cale/\mathscr{A}')
\end{align*}
on $ \P_{F_n}( \overline{\mathscr{A}_n / \fm \mathscr{A}_n})$,
where $p : C \times  \P_{F_n}( \overline{\mathscr{A}_n / \fm \mathscr{A}_n})\to \P_{F_n}( \overline{\mathscr{A}_n / \fm \mathscr{A}_n})$ is the second projection.
Taking $p_*$ of \ref{eq_O(1)_A'_B_n},
we obtain
\begin{align*}
0 \to \calo_{f_n}(1) \to p_*(p_C^* \cale/\mathscr{A}') \to f_n^* \overline{\mathscr{B}_n} \to 0.
\end{align*}
Thus it holds that
\begin{align*}
\pr_{n+1}^* \calo_{F_{n+1}}(1) = \det p_*(p_C^* \cale/\mathscr{A}')  =  \calo_{f_n}(1) \otimes \det  f_n^* \overline{\mathscr{B}_n}  =  \calo_{f_n}(1) \otimes f_n^* \calo_{F_n}(1),
\end{align*}
which is primitive in  $\Pic\P_{F_n}( \overline{\mathscr{A}_n / \fm \mathscr{A}_n}) =\Z   \calo_{f_n}(1) \oplus  \Z f_n^* \calo_{F_n} (1)$.
Hence $\Pic F_{n+1}$ is generated by $ \calo_{F_{n+1}}(1)$, which proves (1) for $n+1$.

To show (2), we determine $\det (\overline{\mathscr{A}_n/\fm \mathscr{A}_n}) \in \Pic F_n$ first.
For a generator $t \in \fm \calo_{C,q}$,
the kernel of $p_C^*  \cale  / p_C^* \fm^{n+1}  \cale \xrightarrow{t \times } p_C^* \cale  / p_C^* \fm^{n+1}  \cale $ on $C \times F_n$
is $p_C^* \fm^{n} \cale  / p_C^* \fm^{n+1} \cale$,
which is contained in $\mathscr{A}_n/ p_C^* \fm^{n+1} \cale$.
Hence we have an exact sequence 
\begin{align*}
0 \to p_C^* \fm^{n} \cale /p_C^* \fm^{n+1}\cale  \to \mathscr{A}_n/p_C^* \fm^{n+1}\cale  \xrightarrow{t \times } \mathscr{A}_n/p_C^* \fm^{n+1}\cale   \to \mathscr{A}_n/\fm \mathscr{A}_n \to 0.
\end{align*}
Taking pushforwards,
we have  an exact sequence 
\begin{align*}
0 \to \overline{p_C^* \fm^{n} \cale /p_C^* \fm^{n+1}\cale}  \to \overline{\mathscr{A}_n/p_C^* \fm^{n+1}\cale}    \to\overline{ \mathscr{A}_n/p_C^* \fm^{n+1}\cale}  
\to \overline{\mathscr{A}_n/\fm \mathscr{A}_n} \to 0
\end{align*}
of locally free sheaves on $F_n$.
Since $\overline{p_C^* \fm^{n} \cale /p_C^* \fm^{n+1}\cale} = (\fm^{n} \cale /\fm^{n+1} \cale) \otimes \calo_{F_n}$ is a trivial bundle of rank $r$,
it holds that $ \det ( \overline{\mathscr{A}_n/\fm \mathscr{A}_n} )=\calo_{F_n}$.

Since $K_{F_n} = \calo_{F_{n}}(-r)$ by induction hypothesis,
we have 
\[ 
K_{\P_{F_n}(\overline{\mathscr{A}_n/\fm \mathscr{A}_n})} = \calo_{f_n}(- r) \otimes f_n^* (K_{F_n} \otimes  \det (\overline{\mathscr{A}_n/\fm \mathscr{A}_n})) = \calo_{f_n} (- r) \otimes f_n^* \calo_{F_{n}}(-r) = \pr_{n+1}^* \calo_{F_{n+1}} (-r).
\]
Thus $K_{F_{n+1}} = {\pr_{n+1}}_* K_{\P_{F_n}(\overline{\mathscr{A}_n/\fm \mathscr{A}_n})}  = \calo_{F_{n+1}} (-r)$,
which proves (2) for $n+1$.

Hence (1) and (2) are proved for any $n \geq 1$.
Since $F_{n,n+1}$ and $F_{n+1}$ are $\Q$-factorial with Picard number two and one respectively,
(3) holds.\\
(4) 
Assume $n \geq 2$.
By \autoref{lem_fiber_F_n,n+1_to_F_n+1},
$Z =\{[\cale \twoheadrightarrow \calb_n ] \in F_n \mid \calb_n \not \simeq \calo_C/\fm^n\}$ is the image of the exceptional divisor of $\pr_{n} : F_{n-1,n}  \to F_n$.
Since the discrepancy of the exceptional divisor of $\pr_{n} : F_{n-1,n}  \to F_n$ is zero by  (2) of this proposition,
$Z$ is contained in the singular locus of $F_n$.
On the other hand, $F_n $ is smooth at $[\cale \twoheadrightarrow \calb_n ] $ if $\calb_n  \simeq \calo_C/\fm^n$ by \cite[Lemma 3.3]{MR4124833}.
Thus $Z$ is  the singular locus of $F_n$.
Since $Z$ is the image of the exceptional divisor of the divisorial contraction $\pr_{n+1}$, $Z$ is irreducible.
By \cite[\S 5]{Birkar:2024aa}, $\codim_{F_n} Z =2$.
\end{proof}

\section{Divisor class groups}\label{sec_div_class/group}

In this section, let $F_n=\Quot_{C}^n (\cale)_q$ be the punctual Quot scheme with $r =\rank \cale \geq 2$ as in the previous section.

\begin{prop}\label{prop_div_class_group}
There exists a prime divisor $H \subset F_n$ such that the divisor class group $\Cl(F_n) $ is generated by the class $[H]$ and 
$n H \sim \calo_{F_n(1)} =\calo_{\Gr(\cale/\fm^n \cale,n) }(1)|_{F_n} $.
\end{prop}

\begin{proof}
If $n=1$, $F_1 \simeq \P^{r-1}$ and hence we can take $H \sim \calo_{\P^{r-1}} (1)$.

Let $n\geq 2$.
We may assume that $\cale = V \otimes \calo_{C}$ for $V=k^r$.
Let $e_1,\dots, e_r$ be the standard basis of $V$.

The smooth locus of $F_n$ is $\{[V \otimes \calo_C \twoheadrightarrow \calb_n ] \in F_n \mid \calb_n \simeq \calo_C/\fm^n\}$ by \autoref{prop_Picard_number} (4).
Hence
the smooth locus is covered by open subsets $U_1,\dots, U_r$ defined by
\[
U_i \coloneqq \{[V \otimes \calo_C \stackrel{\beta}{\twoheadrightarrow} \calo_C/\fm^n] \in F_n \mid \text{the image $\beta(e_i)$ is invertible in $\calo_C/\fm^n$} \}
\]
as explained in \cite[\S 5]{Birkar:2024aa}.
Furthermore, each $U_i$ is isomorphic to $\A^{n(r-1)}$.
For example, we have an isomorphism 
$\A^{n(r-1)} \to U_1$ defined by
\begin{align*}
\beta(e_1) &=1, \\
\beta(e_2) &=a_{2,0} + a_{2,1} t + \cdots + a_{2,n-1} t^{n-1},\\ 
&\vdots \\
\beta(e_r) &=a_{r,0} + a_{r,1} t + \cdots + a_{r,n-1} t^{n-1},
\end{align*}
where $t$ is a generator of the ideal $\fm/\fm^n \subset \calo_C/\fm^n$ and 
$a_{i,j}$'s  are the coordinates of $\A^{n(r-1)}$.
Then $U_1 \setminus U_i =(a_{i,0}  =0) \subset U_1= \A^{n(r-1)}$
and hence $ U_1 \setminus U_i  \simeq \A^{n(r-1)-1}$.
Let 
\begin{align*}
H \coloneqq \overline{U_1 \setminus U_2} \subset F_n
\end{align*}
be the closure of $U_1 \setminus U_2$,
which is a prime divisor of $F_n$.

Consider the composite morphism $\A^{n(r-1)} \simeq  U_1  \subset F_n \hookrightarrow \Gr(V \otimes \calo_C/\fm^n, n)$.
Since $V \otimes \calo_C/\fm^n$ has a basis $\{e_i \otimes t^j \mid 1 \leq i \leq r, 0 \leq j \leq n-1\}$
and $\beta (e_i \otimes t^j) =t^j \beta(e_i)$ for $\beta : V \otimes \calo_C \twoheadrightarrow \calo_C/\fm^n$,
the morphism $\A^{n(r-1)} \hookrightarrow \Gr(V \otimes \calo_C/\fm^n, n)$ is described by the matrix of size $n \times nr$
\begin{align*}
\begin{pmatrix}
A_1 & A_2 & \cdots & A_r
\end{pmatrix},
\end{align*}
where $A_1=E_n$ is the  identity matrix of size $n$ and
\begin{align*}
A_i =\begin{pmatrix}
a_{i,0}& 0 & \cdots  & \cdots &0 \\
a_{i,1}& a_{i,0 } &  \ddots & &\vdots \\
a_{i,2}& a_{i,1 } &a_{i,0} &  \ddots &\vdots \\
\vdots &  \vdots  & & \ddots & 0 \\
a_{i,r-1}& a_{i,r-2 } & \cdots &  \cdots& a_{i,0}
\end{pmatrix}
\end{align*}
for $2 \leq i \leq r$.
For the Pl\"ucker coordinates $p_{1,\dots,n}$ and $p_{n+1,\dots,2n}$ on $\Gr(V \otimes \calo_C/\fm^n, n)$,
we have
\begin{align*}
p_{1,\dots,n}|_{U_1} =\det A_1=1, \quad p_{n+1,\dots,2n}|_{U_1} =\det A_2 = a_{2,0}^n.
\end{align*}
Hence it holds that
\begin{align*}
\Div (p_{1,\dots,n} )|_{U_1} =0, \quad \Div (p_{n+1,\dots,2n})|_{U_1} = nH|_{U_1}.
\end{align*}
By symmetry, we have
$\Div (p_{n+1,\dots,2n})|_{U_2} = 0$. Since $U_2 \cap H =U_2 \cap \overline{U_1\setminus U_2}=\emptyset $,
it holds that $ \Div (p_{n+1,\dots,2n})|_{U_1 \cup U_2} = nH|_{U_1 \cup U_2}$.

Recall that the singular locus $F_n \setminus (U_1 \cup \cdots \cup U_r)$ has codimension two in $F_n$.
For $i \geq 3 $, $U_i \setminus (U_1 \cup U_2) $ is isomorphic to $\A^{n(r-1)-2}$ and hence 
$(U_1 \cup \cdots \cup U_r) \setminus (U_1 \cup U_2)$ has codimension two in $U_1 \cup \cdots \cup U_r$.
Thus $ F_n \setminus (U_1 \cup U_2)$ has codimension two in $F_n$
and hence $\Cl(F_n) = \Cl( U_1 \cup U_2)$.
Since $U_1 \cap H= U_1 \setminus U_2$ and $U_2 \cap H =\emptyset $, we have
$(U_1 \cup U_2) \setminus H =  U_2$.
Then there exists an exact sequence 
\begin{align*}
\Z [H|_{U_1 \cup U_2 }] \to \Cl(U_1 \cup U_2 ) \to \Cl(U_2) \to 0.
\end{align*}
Since $ \Cl(U_2) \simeq \Cl(\A^{n(r-1)}) =0$,
it holds that $\Cl(F_n) = \Cl(U_1 \cup U_2 ) = \Z[H]$.
Since $ \Div (p_{n+1,\dots,2n})|_{U_1 \cup U_2} = nH|_{U_1 \cup U_2}$ and $ \codim_{F_n} (U_1 \cup U_2) =2$,
it holds that $n H = \Div (p_{n+1,\dots,2n})  \sim \calo_{F_n}(1)$.
\end{proof}

\begin{proof}[Proof of \autoref{thm_Picard_number_one}]
(1)-(4) follow from Propositions \ref{lem_embedding_to_Grassmannian} and \ref{prop_Picard_number}.
(5) is nothing but \autoref{prop_div_class_group}.
\end{proof}

\section{The case $r=2$}\label{sec_r=2}

We use the notation in \S \ref{sec_picard_group}.
The purpose of this section is to give a description of the exceptional divisor of $\pr_{n+1} : F_{n,n+1} \to F_{n+1}$ for $r=2$.
Throughout this section, we assume $r=\rank \cale=2$ and hence $ \dim F_n= n(r-1) =n$.

\begin{lem}\label{lem_F_{n-1}_in_F_{n+1}}
For $n \geq 1$, there exists a natural embedding 
\begin{align}\label{eq_embedding_F_n-1_F_n+1}
F_{n-1} \hookrightarrow F_{n+1} \ : \ [\cala_{n-1}] \mapsto [\fm \cala_{n-1}].
\end{align}
\end{lem}

\begin{proof}
Let $
0 \to \mathscr{A}_{n-1} \to p_C^* \cale \to \mathscr{B}_{n-1} \to 0
$
be the universal exact sequence 
on $C \times F_{n-1}$.
Then we have an exact sequence
\[
0 \to \mathscr{A}_{n-1}/\fm \mathscr{A}_{n-1} \to p_C^* \cale /\fm \mathscr{A}_{n-1}  \to \mathscr{B}_{n-1} \to 0.
\]
Since $ \mathscr{A}_{n-1}/\fm \mathscr{A}_{n-1}$  and $\mathscr{B}_{n-1}  $ are flat over $F_{n-1}$ of length $2$ and $n-1$ respectively,
$p_C^* \cale /\fm \mathscr{A}_{n-1}$ is flat over $F_{n-1}$ of length $n+1$.
Since $\det \fm \mathscr{A}_{n-1} =\fm^2 \det \mathscr{A}_n = p_C^* \fm^{n+1} \det \cale $,
the exact sequence $0 \to  \fm \mathscr{A}_{n-1} \to  p_C^* \cale   \to p_C^* \cale /\fm \mathscr{A}_{n-1}  \to 0$ induces the morphism \ref{eq_embedding_F_n-1_F_n+1}.

Furthermore, \ref{eq_embedding_F_n-1_F_n+1} is an embedding since it is the restriction to $F_{n-1} \subset  \Gr(\cale/ \fm^{n-1}\cale, n -1) $
of the embedding
\begin{align*}
 \Gr(\cale/ \fm^{n-1}\cale, n -1)  &\hookrightarrow   \Gr(\cale/ \fm^{n}\cale, n-1 )  \\
 &\simeq \Gr(\fm \cale/ \fm^{n+1}\cale, n-1 ) \hookrightarrow \Gr( \cale/ \fm^{n+1}\cale, n+1),
\end{align*}
induced by the surjection $\cale/ \fm^{n}\cale \to \cale/ \fm^{n-1}\cale $ and an isomorphism  $\cale/ \fm^{n}\cale \simeq \fm \cale/ \fm^{n+1}\cale \subset \cale/ \fm^{n+1}\cale$.
The last embedding is obtained as $\Gr(\fm \cale/ \fm^{n+1}\cale, n-1 ) = \Gr(n+1, \fm \cale/ \fm^{n+1}\cale)  \subset \Gr(n+1, \cale/ \fm^{n+1}\cale) = \Gr( \cale/ \fm^{n+1}\cale, n+1)$.
\end{proof}

\begin{rem}\label{rem_Z_n-1_Z_n+1}
We can check that the embedding \ref{eq_embedding_F_n-1_F_n+1} is the same as the one constructed in \cite[\S 6.4, Proposition 9.1]{Birkar:2024aa}.
\end{rem}

\begin{lem}\label{lem_F_n-1,n_to_F_n,n+1}
For $n \geq 1$, the embedding \ref{eq_embedding_F_n-1_F_n+1} induces an embedding
\begin{align}\label{eq_embedding_incidence}
F_{n-1,n} \hookrightarrow F_{n,n+1}  \ : \ ([\cala_{n-1}],[\cala_n]) \mapsto ([\cala_n], [\fm \cala_{n-1}]).
\end{align}
\end{lem}

\begin{proof}
The embedding \ref{eq_embedding_F_n-1_F_n+1} induces an embedding $ F_{n-1} \times F_n \to F_n \times F_{n+1} : ([\cala_{n-1}],[\cala_n]) \mapsto ([\cala_n], [\fm \cala_{n-1}])$.
If  $([\cala_{n-1}],[\cala_n])  \in F_{n-1,n}$,
it holds that $\cala_{n-1}/\cala_n \simeq \calo_C/\fm$ and hence $ \fm \cala_{n-1} \subset \cala_n$.
Thus $([\cala_n], [\fm \cala_{n-1}])$ is contained in $F_{n,n+1}$.
\end{proof}

The following proposition shows that  the exceptional divisor of $\pr_{n+1} : F_{n,n+1} \to F_{n+1}$ is a $\P^1$-bundle over $F_{n-1} \subset F_{n+1}$.

\begin{prop}\label{prop_exceptional_divisor_r=2}
If $n \geq 1$, $F_{n-1,n} $ embedded in $F_{n,n+1}$ by \ref{eq_embedding_incidence} is the exceptional divisor of $\pr_{n+1} : F_{n,n+1} \to F_{n+1}$.
The restriction  $\pr_{n+1} |_{F_{n-1,n} } : F_{n-1,n} \to F_{n+1}$ coincides with the $\P^1$-bundle  $\pr_{n-1} : F_{n-1,n} \to F_{n-1} \subset F_{n+1}$.
\end{prop}

\begin{proof}
By \ref{eq_embedding_incidence},  $\pr_{n+1}$ maps $([\cala_{n-1}],[\cala_n])  \in F_{n-1,n}$ to $[\fm \cala_{n-1}] \in F_{n+1}$,
which is regarded as $[\cala_{n-1}] \in F_{n-1} \subset F_{n+1}$ under the embedding \ref{eq_embedding_F_n-1_F_n+1}.
Hence the restriction  $\pr_{n+1} |_{F_{n-1,n} } $ coincides with $\pr_{n-1} : F_{n-1,n} \to F_{n-1} \subset F_{n+1}$.
Since $\dim F_{n-1,n} =n =\dim  F_{n,n+1} -1$,
$F_{n-1,n}$ is the exceptional divisor of $\pr_{n,n+1}$.
\end{proof}

\bibliographystyle{amsalpha}
\newcommand{\etalchar}[1]{$^{#1}$}
\providecommand{\bysame}{\leavevmode\hbox to3em{\hrulefill}\thinspace}
\providecommand{\MR}{\relax\ifhmode\unskip\space\fi MR }
\providecommand{\MRhref}[2]{%
  \href{http://www.ams.org/mathscinet-getitem?mr=#1}{#2}
}
\providecommand{\href}[2]{#2}

\end{document}